\documentclass[12pt, reqno]{amsart}
\usepackage[utf8]{inputenc}
\usepackage{times}
\usepackage{graphicx}
\usepackage{amsmath}
\usepackage{bbm}
\usepackage{amssymb,MnSymbol}
\usepackage{indentfirst}
\usepackage{color,xcolor}
\usepackage{amsthm}
\usepackage{centernot}
\usepackage{soul}
\usepackage[top=1in, bottom=1in, left=1in, right=1in]{geometry}
\usepackage[shortlabels]{enumitem}
\usepackage{xcolor}
\usepackage[titlenumbered, ruled]{algorithm2e}
\SetKw{Break}{break}
\SetKw{Return}{return}
\SetKw{EndAl}{end algorithm}
\SetKw{Init}{initialize}
\SetKw{Min}{Min}
\SetKw{Set}{set}

\newtheorem{theorem}{Theorem}[section]

\newtheorem{proposition}[theorem]{Proposition}

\newtheorem{lemma}[theorem]{Lemma}

\theoremstyle{definition}
\newtheorem{example}[theorem]{Example}
\newtheorem{definition}[theorem]{Definition}
\newtheorem{remark}[theorem]{Remark}

\newcommand{\bbF}{\mathbb{F}}

\newcommand{\bbk}{\mathbbm{k}}

\newcommand{\calP}{{\mathcal{P}}}

\newcommand{\newname}{factor }
\newcommand{\newnamecap}{Factor }

\newcommand{\fC}{\Delta_\cap}
\DeclareMathOperator{\fI}{{\operatorname{FI}}}
\DeclareMathOperator{\CF}{{\operatorname{CF}}}

\title{Neural Codes and the Factor Complex}

\author{Alexander Ruys de Perez} 
\address{Department of Mathematics\\Texas A\&M University\\College Station, TX 77843}
\email{alperez@math.tamu.edu}

\author{Laura Felicia Matusevich}
\email{laura@math.tamu.edu}

\author{Anne Shiu}
\email{annejls@math.tamu.edu}

\date{}

\begin{document}

\maketitle

\begin{abstract}
We introduce the \emph{factor complex} of a neural code, and show
how intervals and maximal codewords are captured by the
combinatorics of factor complexes. We use these results to obtain
algebraic and combinatorial characterizations of max-intersection-complete codes, as well as a new combinatorial
characterization of intersection-complete codes.
\end{abstract}

\section{Introduction} \label{sec:intro}
A {\em neural code} on $n$ neurons is a subset of $2^{[n]}$, where $[n]=\{1,2,\dots,n\}$; 
determining which neural codes are convex remains a central open
problem in this area.
The broadest family of codes known to be convex consists of 
\emph{max-intersection-complete codes}, those codes closed under taking intersections of maximal
elements~\cite{cruz2019open,curto2018algebraic}.   
Recently, Curto et
  al.~\cite{curto2018algebraic} asked for an algebraic
signature for max-intersection-complete codes.

Here we answer the question of Curto et al. 
Our main result,
Theorem~\ref{thm:main-intro} below, gives a characterization for when
a code is max-intersection-complete in terms of the canonical form of its neural ideal
(Definitions~\ref{def:neuralIdeal} and~\ref{def:canonicalForm})
and the Stanley--Reisner ideal $I(
\Delta(C))$ of its simplicial complex $\Delta(C)$
(Definitions~\ref{def:StanleyReisner} and~\ref{def:simplicialComplexCode}).

\begin{theorem} \label{thm:main-intro}
A code $C$ on $n$ neurons is max-intersection-complete if and only if for every non-monomial $\phi$
in the canonical form of the neural ideal of $C$, there exists $i\in [n]$ such that 
\begin{itemize}
    \item[(i)] every associated prime of $I(\Delta(C))$ that contains $x_i$ also contains $\phi$, and 
    \item[(ii)] $(1-x_i)\mid \phi$.
    \end{itemize}
\end{theorem}

We remark that
Theorem~\ref{thm:main-intro} can be turned into an
algorithm to verify whether a code is max-intersection-complete.  This
algorithm's runtime is
sub-exponential in the input size, where the input consists of the
maximal codewords of a code $C$ as well as its canonical form
$\CF(J_C)$. On the other hand, the known algorithms for computing $\CF(J_C)$ are
exponential. More details on the computational aspects of
Theorem~\ref{thm:main-intro} can be found in
Section~\ref{sec:algorithm}, which also includes an infinite 
family of codes for which Theorem~\ref{thm:main-intro} is more
efficient at verifying max-intersection-completeness
than brute-force checking of intersections of maximal codewords (see Proposition~\ref{prop:alg}).


To prove Theorem~\ref{thm:main-intro}, which translates a property of
a code to a property of its neural ideal, we introduce a new
combinatorial object, the \emph{\newname complex} of a code.
This is a simplicial complex that, like the neural ideal but unlike
$\Delta(C)$, captures all the combinatorial information in a code $C$.  
We are therefore able to elucidate the relationships among codes,
their factor complexes, and their related ideals (neural ideals and
Stanley--Reisner ideals) -- and then use these results to characterize
being max-intersection-complete in terms of the factor complex.  
Finally, this combinatorial criterion directly translates into an
algebraic criterion, Theorem~\ref{thm:main-intro} above. 

Along the way, we give a new characterization of {\em intersection-complete} codes -- those codes that are closed under taking intersections of codewords. 
Our characterization is combinatorial, via the factor
complex, in contrast to a prior algebraic characterization through
the neural ideal~\cite{curto2018algebraic}.   
Indeed, we expect in the future that the factor complex may help us understand more properties of neural codes.

Our work fits into the literature on neural codes as follows. 
Like previous works, we are motivated by the question of convexity in
neural codes~\cite{what-makes,just-convex, jeffs2018, jeffs2019sparse, LSW,
  zvi-yan, williams},  
with a specific interest in using neural ideals to study
convexity~\cite{curto2013neural, GB, new-alg, HMO, IKR,  morvant}.
Also, our factor complexes are motivated by the closely related {\em
  polar complexes} introduced recently by G\"unt\"urk\"un et
al.~\cite{gunturkun2017polarization} (see
also~\cite{christensen,IKR}). 

\subsection*{Outline}
This article is organized as follows.
Section~\ref{sec:bkrd} contains background material, and 
Section~\ref{sec:main-results} gives our main results.
In Section~\ref{sec:reln-to-ideals},
we prove relationships among codes, their factor complexes, and their
neural or Stanley-Reisner ideals, and Section~\ref{sec:polar} relates
factor complexes and polar complexes.

\subsection*{Author contributions}
The first author is the main contributor to this article.

\subsection*{Acknowledgements
}
We are grateful to Carina Curto, Luis Garc\'{\i}a-Puente, and Alex Kunin for inspirational
discussions. 
AS was partially supported by NSF grant DMS-1752672.
We thank the referee for helpful suggestions which improved this work.

\section{Background} \label{sec:bkrd}
Throughout this article, $C$ is a neural code on $n$ neurons, that is,
a subset of $2^{[n]}$, where $[n]=\{1,2,\dots,n\}$. 
Elements of $C$
are called \emph{codewords}, and may be represented as subsets of
$[n]$ or as $n$-tuples of zeros and ones, where a $1$ in position $i$
indicates that $i$ belongs to the codeword. 

Given $c \subset d \subset [n]$, the \emph{Boolean
  interval} between $c$ and $d$ is 
\begin{equation*}
\label{eqn:booleanInterval}
[c,d]:=\{ w \in 2^{[n]} \mid c \subset w \subset d \}.
\end{equation*}

The \emph{complement} of a code $C$ on $n$ neurons is the code
\begin{equation}
\label{eqn:complementCode}
C':= 2^{[n]} \smallsetminus C .
\end{equation}

\noindent
{\bf Convention.} In this article, we assume that $\emptyset \subsetneq C \subsetneq 2^{[n]}$, so that 
the neural ideals (defined below) of $C$ and $C'$ have primary decompositions.

\begin{definition}
\label{def:intervalsOfC}
Let $C$ be a code. The \emph{intervals} of $C$ are the Boolean
intervals contained in $C$. The \emph{maximal intervals} of $C$ are
the intervals of $C$ that are maximal with respect to inclusion.
\end{definition}

\begin{example} \label{ex:code}
For the code $C=\{ \emptyset, 2,3, 12, 13 \} = \{000, 010, 001, 110, 101 \}$, the maximal intervals 
are $[\emptyset, 2]$, $[\emptyset, 3]$, $[2,12]$, and $[3,13]$.
\end{example}

\subsection{Neural ideals and the canonical form}
The main reference for
this section is~\cite{curto2013neural}.

We denote by $\bbF_2$ the field with two elements, and let
$R=\bbF_2[x_1,\dots,x_n] = \bbF_2[x]$. A \emph{pseudomonomial} is a polynomial
$\prod_{i\in \sigma}x_i\prod_{j\in \tau}(1-x_j) \in R$, where $\sigma,\tau
\subset [n]$ are disjoint. A \emph{pseudomonomial ideal} is an ideal
generated by pseudomonomials.
If $c \in 2^{[n]}$, the pseudomonomial
\begin{equation}
\label{eqn:pseudoMonomial}
\phi_c:= \prod_{i \in c} x_i \prod_{j \in [n] \smallsetminus c}
    (1-x_j)
\end{equation}
is called the \emph{indicator polynomial} of $c$. 

\begin{definition}
\label{def:neuralIdeal}
The \emph{neural ideal} $J_C$ of a code $C$ is the (pseudomonomial) ideal generated by the
indicator polynomials of its non-codewords; in symbols,
\[
J_C := \langle \phi_c \mid c \in C'\rangle.
\]
\end{definition}
Note that, using the convention that
$n$-tuples of zeros and ones represent codewords, the zero-set of
$J_C$ is $C$. In other words, the code $C$ and its neural ideal 
contain the same information. Moreover, any
ideal generated by pseudomonomials is the neural ideal of a code~\cite[Theorem 2.1]{jeffs2016convexity}.

The neural ideal $J_C$ has a unique irredundant decomposition
\begin{equation}
\label{eqn:primdec}
J_C=\bigcap\limits_{h=1}^g P_h, 
\end{equation}
where each $P_h$ is a
pseudomonomial ideal that is
prime~\cite[Proposition 6.8]{curto2013neural}.
In particular, $J_C$ is a radical ideal. We remark that a
pseudomonomial ideal $P$ is prime if and only if it is of the form
\begin{equation}
\label{eqn:prime}
P = \langle \{x_i \mid i \in \sigma\} \cup \{(1-x_j) \mid j \in
\tau\}\rangle \quad
\text{~for~} \sigma,\tau \text{~disjoint~subsets~of~}  [n].
\end{equation}

\begin{definition}
\label{def:canonicalForm}
Let $J \subset R$ be a pseudomonomial ideal. A
pseudomonomial in $J$ is \emph{minimal} if it is minimal with
respect to divisibility among all pseudomonomials in $J$. The
\emph{canonical form} of $J$ is the set $\CF(J)$ of all minimal
pseudomonomials of $J$.
\end{definition}

The canonical form of a pseudomonomial ideal is a generating set for the ideal~\cite{curto2013neural}.

\begin{example}[Example~\ref{ex:code}, continued] \label{ex:code-2}
The complement of the code $C=\{ \emptyset, 2,3, 12, 13 \}$
is $C'=\{1, 23, 123\}$. Thus, the neural ideal of $C$ is $J_C=\langle x_1(1-x_2)(1-x_3), ~x_2x_3(1-x_1),~ x_1x_2x_3 \rangle$, and the canonical form is $\CF(J_C)=\{x_1(1-x_2)(1-x_3),~  x_2x_3 \}$.
\end{example}

\subsection{Polarization and squarefree monomial ideals}
Let $S = \bbF_2[x_1,\dots,x_n,y_1,\dots,y_n] = \bbF_2[x,y]$.

The idea of using $y_i$ to encode $1-x_i$ is well known (see, for instance,~\cite{jarrah,veliz-cuba}).  In the context of neural ideals, the
following construction was introduced in~\cite{gunturkun2017polarization}.
\begin{definition}
\label{def:polarization}
The \emph{polarization} of a pseudomonomial $\phi=\prod_{i\in \sigma}
x_i \prod_{j\in \tau}(1-x_j) \in R$ is 
\begin{equation*}
\label{eqn:polarizationOfPhi}
\calP(\phi) := \prod_{i\in\sigma} x_i \prod_{j \in \tau}y_j \in S.
\end{equation*}
If $J \subset R$ is a pseudomonomial ideal, the
\emph{polarization} of $J$ is the ideal in $S$ obtained by polarizing the
pseudomonomials in the canonical form of $J$, that is,
\begin{equation*}
\label{eqn:polarizationOfJ}
\calP(J) := \langle \calP(\phi) \mid \phi \in \CF(J) \rangle
\subset S.
\end{equation*}
\end{definition}
Note that the polarization of a pseudomonomial ideal is a
\emph{squarefree} monomial ideal in $S$, that is,
an ideal generated by monomials that are not divisible by the squares
of the variables (so, $\calP(J)$ is radical). We recall the relationship between squarefree
monomial ideals and simplicial complexes.

\begin{definition}
\label{def:StanleyReisner}
Let $\Delta$ be a simplicial complex on $[n]$, and let $\bbk$ be a
field. The \emph{Stanley--Reisner ideal} of $\Delta$ is
\begin{equation*}
\label{eqn:StanleyReisner}
I(\Delta) := \langle \prod_{i\in \sigma} x_i \mid \sigma \notin \Delta
\rangle \subset \bbk[x_1,\dots,x_n].
\end{equation*}
\end{definition}

The ideal $I(\Delta)$ is radical, with prime decomposition
\begin{equation}
\label{eqn:SRDecomp}
I(\Delta) = \bigcap_{\sigma \in \textrm{Facets}(\Delta)} \langle x_i \mid
i \notin \sigma \rangle.
\end{equation}
It follows that $\Delta$ can be recovered from $I(\Delta)$. In
fact,~\eqref{eqn:SRDecomp} can be used to conclude that
any squarefree monomial ideal is the Stanley--Reisner ideal of some
simplicial complex.

\begin{definition}
\label{def:simplicialComplexCode}
The {\em simplicial complex of a code} $C$ is $\Delta(C)$, the
smallest simplicial complex containing $C$. Its Stanley--Reisner ideal 
is denoted by $I(\Delta(C)) \subset R=\bbF_2[x]$. 
\end{definition}

It is a fact that $I(\Delta(C))$ is
generated by the monomials in $\CF(J_C)$~\cite[Lemma 4.4]{curto2013neural}. 

\begin{example}[Example~\ref{ex:code-2}, continued] \label{ex:code-3}
For $C=\{ \emptyset, 2,3, 12, 13 \}$, 
the simplicial complex $\Delta(C)$ has two facets, $12$ and $13$.  The corresponding Stanley--Reisner ideal 
is $I(\Delta(C))= \langle x_2 x_3 \rangle$, which is generated by the unique monomial in the canonical form
$\CF(J_C)=\{x_1(1-x_2)(1-x_3),~  x_2x_3 \}$.
\end{example}

In this article, we work with squarefree monomial ideals in
$S=\bbF_2[x,y]$ that arise from polarization. 
In order to construct their corresponding simplicial complexes, we use
$\{1,\dots,n,\overline{1},\dots,\overline{n}\}$ as a vertex set, with
the understanding that $x_i$ corresponds to $i$, and $y_i$ corresponds
to $\overline{i}$. If $B \subset [n]$, we denote $\overline{B} = \{
\overline{i} \mid i \in B\}$. In particular,
\[
\overline{[n]}=\{\overline{1},\dots,\overline{n}\} \quad \textrm{and} \quad
[n]\cup\overline{[n]} = \{1,\dots, n,\overline{1},\dots,\overline{n}\}.
\]
We always use overline notation to denote subsets of $\overline{[n]}$;
this is justified, as any subset of $\overline{[n]}$ is of the form $\overline{B}$ for
some $B \subset [n]$.

\begin{remark}
As noted above, the ideals that are associated to codes (the neural ideal
$J_C$, the ideal $I(\Delta(C))$, and later the factor ideal $\fI(C)$)
are \emph{radical ideals}, that is, they can be expressed as intersections of prime ideals. 
We emphasize that the sets of associated
primes, minimal primes, and primary components of a radical ideal all coincide.
\end{remark}

\section{Main results} \label{sec:main-results}

In this section we introduce a new combinatorial tool to study neural
codes: the factor complex
(Definition~\ref{def:new-cpx}), and state our four main results.  
Theorems~\ref{thm:relationships-1}
and~\ref{thm:relationships-2} summarize the relationships among codes,
their factor complexes, and their related ideals (neural ideals and
Stanley--Reisner ideals). 
These results are used to prove Theorems~\ref{thm:intersection-complete}
and~\ref{thm:max-intersection-complete}, 
which characterize intersection-complete codes and max-intersection-complete codes in two ways: combinatorially and algebraically.

\begin{definition} 
\label{def:new-cpx}
Let $C$ be a code on $n$ neurons, and recall the primary
decomposition of the neural ideal $J_C$ given in~\eqref{eqn:primdec}.
The \emph{\newname ideal} of $C$
is obtained by polarizing the components 
of $J_C$, namely,
\[
\fI(C) := \bigcap\limits_{h=1}^g \calP(P_h).
\]
The \emph{\newname complex} $\fC(C)$ of $C$ is 
the 
simplicial complex on $[n]\cup\overline{[n]}$
whose Stanley--Reisner ideal is $\fI(C)$.
A face of $\fC(C)$ is \emph{defective} if it contains neither $i$ nor
$\overline{i}$ for some $i\in [n]$ (we think of $i$ as
a defect, or flaw); faces that are not defective are called \emph{effective}.
We say that $\overline{B} \subset \overline{[n]}$ is a {\em prime-set} of
$\fC(C)$  if $[n]\cup \overline{B} \notin \fC(C)$,
and $\overline{B}$ is furthermore {\em minimal} if 
$\overline{B}$ is minimal with respect to inclusion among 
prime-sets. Lemma~\ref{lemma:prime-set} gives the reason why we chose
this terminology.
\end{definition}

\begin{example}[Example~\ref{ex:code-3}, continued] \label{ex:code-4}
For $C' = \{1,23,123 \}$,
the neural ideal decomposes as follows:
$$
J_{C'} = \langle (1-x_1)(1-x_3),~ (1-x_1)(1-x_2),~x_2(1-x_3),~x_3(1-x_2)\rangle
=
\langle x_2, x_3, ~1-x_1 \rangle \cap \langle 1-x_2,~ 1-x_3 \rangle.
$$
The factor ideal is therefore
$$
FI(C') = 
\langle x_2, x_3, y_1\rangle \cap \langle y_2, y_3 \rangle,$$
and so the two facets of the factor complex 
$\Delta_{\cap}(C')$ are
 $1\bar{2}\bar{3}$ and $123\bar{1}$ (both are effective). The minimal prime-sets of 
 $\Delta_{\cap}(C')$ are
$\{\bar{2}\}$ and $\{\bar{3}\}$.
 \end{example}

\begin{theorem}[Codes, factor complexes, and neural ideals] \label{thm:relationships-1}
Let $C$ be a code on $n$ neurons, and $C'$ its complement code defined
in~\eqref{eqn:complementCode}.
The following two maps are bijections:
\begin{align*}
\begin{array}{ccccc}
    \{\textrm{pseudomonomials~in~}J_{C'} \}
        & \leftarrow &
    \{\textrm{intervals~in~}C\}
     &\to & 
    \{\textrm{effective~faces~of~}\fC(C) \}
    \\
        \prod_{i \in c} x_i \prod_{j \in [n] \smallsetminus d} (1-x_j)
        & \leftmapsto &
    [c,d]
     &\mapsto & 
    d \cup \overline{[n] \smallsetminus c}\\
\end{array}
\end{align*}
Moreover, every facet of $\fC(C)$ is effective, and the
following are equivalent:
\begin{enumerate}[(1)]
\item $[c,d]$ is a maximal interval in $C$, 
\item $\prod_{i \in c} x_i \prod_{j \in [n] \smallsetminus d} (1-x_j)
  \in \CF(J_{C'})$, and
\item $d \cup \overline{[n]
  \smallsetminus c}$ is a facet of $\fC(C)$.  
\end{enumerate}
\end{theorem}

\begin{theorem}[Codes, factor complexes, and Stanley--Reisner ideals] \label{thm:relationships-2}
Let $C$ be a code on $n$ neurons, with complement code $C'$ and factor
complex $\fC(C)$.  The following two maps are bijections: 
\begin{align*}
\begin{array}{ccccc}
    \{\textrm{minimal~primes~of~}I(\Delta(C))\}
        & \leftarrow &
    \{\textrm{maximal~codewords~of~}C \}
     &\to & 
    \left \{ \begin{array}{c}\textrm{minimal prime-sets} \\ \textrm{of
               } \fC(C') \end{array} \right\}
    \\
       \langle x_i \mid i \in [n] \smallsetminus M \rangle 
        & \leftmapsto &
    M
     &\mapsto & 
     \overline{[n] \smallsetminus M}
    \\
\end{array}
\end{align*}
\end{theorem}

The proofs of Theorems~\ref{thm:relationships-1}
and~\ref{thm:relationships-2} are postponed until Sections~\ref{sec:reln-to-neural-ideals}
and~\ref{sec:reln-to-SR-ideals}, respectively.

\begin{example}[Example~\ref{ex:code-4}, continued] \label{ex:code-5}
According to 
Theorem~\ref{thm:relationships-1}, 
the facets $1\bar{2}\bar{3}$ and $123\bar{1}$ 
of 
 $\Delta_{\cap}(C')$ 
correspond to the two
maximal intervals of $C'$, 
$[1,1]$ and $[23,123]$, respectively,
and also 
to the two pseudomonomials in  
$\CF(J_C)$, 
namely, $x_1(1-x_2)(1-x_3)$ and $x_2x_3$, respectively.
 
Similarly, Theorem~\ref{thm:relationships-2}
implies that the minimal prime-sets
 $\{\bar{2}\}$ and $\{\bar{3}\}$ 
of 
 $\Delta_{\cap}(C')$ 
correspond to the minimal primes 
$\langle x_2\rangle$ and $\langle x_3 \rangle$
of $I(\Delta(C))= \langle x_2 x_3 \rangle$ 
and also to the maximal codewords $13$ and $12$ of $C$.
\end{example}

The following result translates the algebraic characterization of
intersection-complete codes from~\cite{curto2018algebraic} into a
new 
combinatorial criterion.
\begin{theorem}[Intersection-complete codes]
\label{thm:intersection-complete}
Let $C$ be a code on $n$ neurons with neural ideal $J_C$, and let $C'$ be the complement code of $C$
with \newname complex 
$\fC(C')$. The following are equivalent:
\begin{enumerate}[(1)]
    \item $C$ is intersection-complete, \label{ic1}
    \item every pseudomonomial $\prod\limits_{i\in\sigma}
x_i\prod\limits_{j\in\tau}(1-x_j) $ in $ \CF(J_C)$  \label{ic2}
satisfies $|\tau|\leq 1$, and
    \item every facet $F$ of $\fC(C')$ satisfies $\left|F \cap [n]
      \right| \geq n-1$. \label{ic3}
\end{enumerate}
\end{theorem}

\begin{proof}
The equivalence between~\emph{\ref{ic1}} and~\emph{\ref{ic2}} is~\cite[Theorem 1.9]{curto2018algebraic}. 
By Theorem~\ref{thm:relationships-1}, 
 $\prod\limits_{i\in\sigma} x_i\prod\limits_{j\in\tau}(1-x_j)$ belongs
 to the canonical form
 of $J_C$ 
if and only if $F=[n] \smallsetminus \tau \cup \overline{[n]
  \smallsetminus \sigma}$ is a facet of $\fC(C')$. 
Thus, the condition $|\tau|\leq 1$ is equivalent to $\left|F \cap [n]
\right| \geq n-1$, and so~\emph{\ref{ic2}} is equivalent to~\emph{\ref{ic3}}. 
\end{proof}

The following result is an expanded version of
Theorem~\ref{thm:main-intro}.

\begin{theorem}[Max-intersection-complete codes]\label{thm:max-intersection-complete}
Let $C$ be a code on $n$ neurons with neural ideal $J_C$, and let $C'$ be the complement code of $C$
with \newname complex
$\fC(C')$. The following are equivalent:
\begin{enumerate}[(1)]
    \item $C$ is max-intersection-complete, \label{mic1}
    \item  \label{mic2} for every facet $F$ of $\fC(C')$ that does \underline{not}
      contain $[n]$, there exists $i \in [n]$ such that  
\begin{enumerate}[(i)]
        \item every minimal prime-set of $\fC(C')$ that contains
          $\overline{i}$ also contains some $\overline{j}$ such that
          $\overline{j} \notin F$, and \label{mic2ii}
        \item  $i \notin F$,    \label{mic2i}
\end{enumerate} 
    \item for every $\phi\in \CF(J_C)$ that is \underline{not} a monomial, there exists $i\in [n]$ such that 
\begin{enumerate}[(i)]
    \item every minimal prime of $I(\Delta(C))$ that contains
      $x_i$ also contains $\phi$, and \label{mic3ii}
    \item $(1-x_i)\mid \phi$. \label{mic3i}
\end{enumerate} \label{mic3}
\end{enumerate}
\end{theorem}

\begin{proof}
We begin by proving~\emph{\ref{mic2}$\Leftrightarrow$\ref{mic3}}.
By Theorem~\ref{thm:relationships-1}, 
$\phi = \prod_{i \in c} x_i \prod_{j \in [n] \smallsetminus d} (1-x_j)
\in \CF(J_C)$ if and only if $F=d \cup \overline{[n]\smallsetminus c}$
is a facet of $\fC(C')$.  
Furthermore, $\phi$ is a non-monomial exactly when $d \not\supseteq [n]$,
if and only if $F$ does not contain $[n]$. 
Thus, by inspection of $\phi$ and $F$, \emph{\ref{mic2}\ref{mic2i}} is equivalent
to~\emph{\ref{mic3}\ref{mic3i}}, and so we need only show
\emph{\ref{mic2}\ref{mic2ii}$\Leftrightarrow$\ref{mic3}\ref{mic3ii}}.

By Theorem~\ref{thm:relationships-2}, the prime ideal
$P=\langle x_j \mid j \in B \rangle$ is associated to 
$I(\Delta(C))$
if and only if
$\overline{B}$ is a minimal prime-set of $\fC(C')$. 
Thus, $x_i \in P$ exactly when $\overline{i} \in \overline{B}$.  Next, 
it is straightforward to check that $P$ contains $\phi = \prod_{i \in c} x_i \prod_{j \in [n] \smallsetminus d} (1-x_j)$ if and only if $B \cap c \neq \emptyset$. 
As $\phi$ corresponds to the facet $F=d \cup
\overline{[n]\smallsetminus c}$ of $\fC(C')$,
it follows that $P$ contains $\phi$ if and only if $\overline{j}
\notin F$ for some $\overline{j} \in \overline{B}$.  This concludes
the proof of~\emph{\ref{mic2}$\Leftrightarrow$\ref{mic3}}.

We set up notation needed to prove~\emph{\ref{mic1}$\Leftrightarrow$\ref{mic2}}. 
Let $\overline{B}_1, \overline{B}_2,\dots,\overline{B}_u$ be the
minimal prime-sets of $\fC(C')$.
By Theorem~\ref{thm:relationships-2}, the maximal codewords of $C$ are
$m_1 = [n]\smallsetminus B_1,\dots,m_u = [n]\smallsetminus B_u $.

We claim that~\emph{\ref{mic2}} is equivalent to the following:
\emph{
\begin{enumerate}[(2')]
    \item \label{mic2p}
    for every facet $F$ of $\fC(C')$ that does not contain $[n]$, 
    \begin{align} \tag{$\star$}\label{eq:max-intersection-complete}
            ([n]\smallsetminus \bigcup\limits_{v\in H_F} B_{v})\not\subset F,
    \end{align}
    where 
    $$H_{F} :=\{v\in [u]\ |\ \overline{B}_v\subset F \}.$$
\end{enumerate}    
}
Indeed, condition~\eqref{eq:max-intersection-complete} 
states that there exists $i \in [n]$ such that 
$i\notin F$ and $\overline{i}$ is \underline{not} in any minimal prime-set $\overline{B}_v
\subset \{\overline{1}, \overline{2}, \dots, \overline{n}\}$ for
which $\overline{B}_v \subset F$.  This latter condition
exactly matches~\emph{\ref{mic2}\ref{mic2ii}}.  Hence, our claim holds, and we may complete
this proof by showing~\emph{\ref{mic1}$\Leftrightarrow$\ref{mic2p}}.

$(\Leftarrow)$ We prove the contrapositive. Suppose 
that the intersection of maximal codewords
$c = \bigcap\limits_{v\in V} m_{v}$ (for some $\varnothing \neq V \subset [u]$)
is not in $C$, that is, $c \in C'$. 
By Theorem~\ref{thm:relationships-1},
$c \cup \overline{[n]\smallsetminus c}$ is a face of $\fC(C')$. Note that 
\begin{equation}
\overline{[n]\smallsetminus c} = \overline{[n]\smallsetminus
  \bigcap\limits_{v\in V} m_v} = \bigcup\limits_{v\in
  V}\overline{[n]\smallsetminus m_v} = \bigcup\limits_{v\in V}
\overline{B}_v.\label{eq:1} 
\end{equation} 
Let $F$ be a facet of $\fC(C')$ containing $c \cup
\overline{[n]\smallsetminus c}$. It follows from~\eqref{eq:1} that $F$
contains the union of minimal prime-sets $\bigcup\limits_{v\in V}
\overline{B}_v$, which implies that $F$ does not contain
$[n]$ (as, otherwise, each $\overline{B}_v \cup [n]$ is contained in $F$ and hence 
is a face of $\fC(C')$, contradicting the fact that $B_v$ is a prime-set).
Since $F \supset \overline{[n]\smallsetminus c} =
\bigcup\limits_{v\in V} \overline{B}_v$, we have that 
$V \subset  H_F$. 
Therefore, $[n]\smallsetminus \bigcup\limits_{v\in H_F}B_v \subset
[n]\smallsetminus \bigcup\limits_{v\in V} B_v=c$, where the equality
comes from~\eqref{eq:1}.
We conclude that $F$ is a facet of $\fC(C')$ not containing $[n]$ such that
$([n]\smallsetminus \bigcup\limits_{v\in H_F} B_v)\subset c
\subset ( c \cup \overline{[n]\smallsetminus c}) \subset F$.  

$(\Rightarrow)$ Suppose $C$ is max-intersection-complete. Let $F$ be a facet of $\fC(C')$ that does not contain $[n]$. Set $c:= [n]\smallsetminus \bigcup\limits_{v\in H_F}{B_v}$. Our goal is to show that $c\not\subset F$.

We accomplish this by proving two facts. First, that $c\cup
\overline{[n]\smallsetminus c}$ is not a face of $\fC(C')$, and
second, that $\overline{[n]\smallsetminus c} =  \bigcup\limits_{v\in
  H_F}\overline{B}_v$. 
The first fact implies that $c\cup \overline{[n]\smallsetminus
  c}\not\subset F$ and the second yields $\overline{[n]\smallsetminus c}
\subset F$. Our desired relation $c\not\subset F$ will then
follow. 
 
 For the first fact, recall that $[n]\smallsetminus
 B_v = m_v$. Therefore, 
\[
c = [n]\smallsetminus \bigcup\limits_{v\in H_F} B_v =
\bigcap\limits_{v\in H_F} [n]\smallsetminus B_v =
\bigcap\limits_{v\in H_F} m_v,
\] 
so $c$ is the intersection of maximal codewords. As $C$
is max-intersection-complete, $c \in C$, and thus $c\not\in C'$. Now
Theorem~\ref{thm:relationships-1} implies that $c\cup
\overline{[n]\smallsetminus c}\not\in \fC(C')$. 
 
 For the second fact, 
$
\overline{[n]\smallsetminus c} = \overline{[n]\smallsetminus
  ([n]\smallsetminus \bigcup\limits_{v\in H_F}B_v}) =
\overline{\bigcup\limits_{v\in H_F}B_v} = \bigcup\limits_{v\in H_F} \overline{B}_v.
$
\end{proof}

\begin{example}[Example~\ref{ex:code-5}, continued] \label{ex:code-6}
The code $C=\{ \emptyset, 2,3, 12, 13 \}$ is neither intersection-complete nor max-intersection-complete (as $1 = 12 \cap 13 \notin C$).  We can read this information from Theorems~\ref{thm:intersection-complete} and~\ref{thm:max-intersection-complete}, as follows.
For non-intersection-completeness, this can be seen in two ways: first, the pseudomonomial  $x_1(1-x_2)(1-x_3)$
 is in the canonical form of $J_C$, and, second, the intersection of the facet $1 \bar{2} \bar{3}$ with $123$ has size 1, rather than 2 or 3.  
 
 For non-max-intersection-completeness, recall that 
 the minimal prime-sets of $\fC(C')$
are
 $\{\bar{2}\}$ and $\{\bar{3}\}$ (equivalently, the minimal primes of $I(\Delta(C))$ are $\langle x_2 \rangle$ and $\langle x_3 \rangle$). Now, $1 \bar{2} \bar{3}$ is a facet of $\fC(C')$ 
that does not contain 123, 
but for $i\in \{1,2,3\}$, either part \emph{\ref{mic2}\ref{mic2ii}} of Theorem~\ref{thm:max-intersection-complete} is violated (when $i=2,3$)
or part \emph{\ref{mic2}\ref{mic2i}} is violated (when $i=1$).
Alternatively, $\CF(J_C)$ contains the non-monomial $x_1(1-x_2)(1-x_3)$, but
for $i\in \{1,2,3\}$, either part \emph{\ref{mic3}\ref{mic3ii}} of Theorem~\ref{thm:max-intersection-complete} is violated (when $i=2,3$)
or part \emph{\ref{mic3}\ref{mic3i}} is violated (when $i=1$).  Thus, $C$ is not max-intersection-complete.
\end{example}

\section{\newnamecap complexes, neural ideals, and codes} \label{sec:reln-to-ideals}
In this section, we prove Theorems~\ref{thm:relationships-1} and~\ref{thm:relationships-2}.

\subsection{Proof of Theorem~\ref{thm:relationships-1}} \label{sec:reln-to-neural-ideals}
We wish to prove that the following maps are bijections:
\begin{align*}
\begin{array}{ccccc}
    \{\text{pseudomonomials~in~}J_{C'} \}
        & \xleftarrow{\alpha} &
    \{\text{intervals~in~}C\}
     &\xrightarrow{\beta} & 
    \{\text{effective~faces~of~}\fC(C) \}
    \\
        \prod_{i \in c} x_i \prod_{j \in [n] \smallsetminus d} (1-x_j)
        & \leftmapsto &
    [c,d]
     &\mapsto & 
    d \cup \overline{[n] \smallsetminus c}\\
\end{array}
\end{align*}
The fact that $\alpha$ is a bijection is straightforward
from~\cite[Lemma 5.7]{curto2013neural}. To show that $\beta$ is a
bijection, we need to better understand the factor ideal and factor
complex of $C$.

\begin{lemma}
\label{lemma:factorIdealPseudomonomials}
Let $C$ be a code
with neural ideal $J_C$, and let
$\phi$ be a pseudomonomial.
Then $\phi\in J_C$ if and only if 
$\calP(\phi)\in \fI(C)$.
\end{lemma}

\begin{proof}
Recall the decomposition $J_C = \bigcap^g_{h=1} P_h$
from~\eqref{eqn:primdec}. 
Hence, $\phi \in J_C$ if and only if $\phi \in
P_h$ for all $h$. Given the form~\eqref{eqn:prime} of each
component $P_h$, 
it is straightforward to check that 
$\phi \in P_h$ is equivalent to $\calP(\phi)\in \calP(P_h)$.
Thus, as $\fI(C)= \bigcap \mathcal{P}(P_h)$, 
the desired result follows.
\end{proof}

Our next results shows how to use the factor complex of a code to read
off its codewords.
\begin{lemma}\label{codewordfacelem}
Let $C$ be a code on $n$ neurons. Then $c\in 2^{[n]}$ is a codeword 
of $C$
if
and only if $c\cup \overline{[n]\smallsetminus c}$ is a face of
$\fC(C)$. 
\end{lemma}
\begin{proof}
By~\cite[Lemma~3.2]{curto2013neural}, $c\in C$ if and only if
$\phi_c=\prod_{i\in c}x_i\prod_{j\notin c} (1-x_j) \notin
J_C$. This is
equivalent to $\calP(\phi_c) \notin \fI(C)$ by Lemma~\ref{lemma:factorIdealPseudomonomials}.
Since $\fI(C)$ is the
Stanley--Reisner ideal of $\fC(C)$, we have that
$\calP(\phi_c) \notin \fI(C)$ exactly when 
$c\cup \overline{[n]\smallsetminus c}$ is a face of
$\fC(C)$, which concludes the proof.
\end{proof}

We now extend Lemma~\ref{codewordfacelem} to
show how to extract the intervals of $C$ from its
factor complex.

\begin{lemma} (Interval-Face Correspondence)\label{IFCorr}
Let $C$ be a code on $n$ neurons, and let $c,d\in 2^{[n]}$. 
Then $[c,d]\subset C$ if and only if $d\cup
\overline{[n]\smallsetminus c}$ 
is a face of $\fC(C)$. 
\end{lemma}
\begin{proof}
$(\Leftarrow)$ Suppose $d\cup \overline{[n]\smallsetminus c}$ is a
face of $\fC(C)$, and let $w\in [c,d]$. Then $w \cup
\overline{[n]\smallsetminus w} \subset d \cup
\overline{[n]\smallsetminus c}$ is a face of $\fC(C)$ and thus $w\in
C$ by Lemma~\ref{codewordfacelem}. 

$(\Rightarrow)$ We now assume that $d\cup \overline{[n]\smallsetminus c}$
is not a face of $\fC(C)$ and show that $[c,d]$ is not an interval of $C$. 
As $\fI(C)$ is the Stanley--Reisner
ideal of $\fC(C)$, the decomposition~\eqref{eqn:SRDecomp} implies that the ideal
\[\left\langle\{x_i\ |\ i\not\in d\cup \overline{[n]\smallsetminus
    c}\}\cup \{y_j\ |\ \bar{j}\not\in d\cup
  \overline{[n]\smallsetminus c}\}\right\rangle = \big\langle \{x_i \mid i
\in [n]\smallsetminus d \} \cup \{ y_j \mid j \in c\} \big\rangle
\]
is \underline{not} associated to $\fI(C)$, and therefore the following ideal 
is \underline{not} associated to $J_C$:
\begin{equation}
\label{eqn:notAssociated}
\big\langle
\{x_i \mid i \in [n]\smallsetminus d \}\cup \{(1-x_j) \mid j \in c\}\big\rangle.
\end{equation}
 Thus, as $ \CF(J_C)$ is a generating set for $J_C$,
 there exists a pseudomonomial $\phi =
\prod_{i\in
  \sigma} x_i \prod_{j\in \tau }(1-x_j)$ in $ \CF(J_C)$
that is not in the ideal~\eqref{eqn:notAssociated}, and so $\sigma
\subset d$ and $\tau \subset [n]\smallsetminus c$. Note that the
indicator pseudomonomial
$\phi_{c\cup \sigma}$ 
is in $J_C$, as it is divisible by $\phi$. 
We conclude that $\sigma
\cup c \in [c,d] \smallsetminus C$, and so $[c,d] \not\subset C$.
\end{proof}

We can now better understand the facets of $\fC(C)$.

\begin{lemma}\label{propFulRep}
Let $C$ be a code on $n$ neurons. Every facet of $\fC(C)$ is effective.
\end{lemma}

\begin{proof}
By~\eqref{eqn:SRDecomp}, the facets of $\fC(C)$ correspond to
associated primes of $\fI(C)$, which are polarizations of associated
primes of $J_C$. Since the latter primes cannot contain both $x_\ell$ and
$1-x_\ell$, it follows that the former primes cannot contain both $x_\ell$
and $y_\ell$, which concludes the proof.
\end{proof}

\begin{proof}[Proof of Theorem \ref{thm:relationships-1}]
By~\cite[Lemma 5.7]{curto2013neural}, the map $\alpha$ is a bijection,
and the correspondence between minimal pseudomonomials and maximal
intervals follows from the fact for any two intervals $M_1$ and $M_2$
of $C$, we have $M_1\subset M_2$ if and only if $\alpha(M_2)\mid
\alpha(M_1)$. By~Lemma \ref{IFCorr}, 
plus the fact that effective faces have the form $d \cup \overline{[n] \smallsetminus c}$
for some $c \subset d$,
the map $\beta$ is also a
bijection. Lemma~\ref{propFulRep} states that all facets of
$\fC(C)$ are effective, and thus for each facet $F$ we have $F =
\beta(M)$ for some interval $M$ of $C$. The correspondence between
facets and maximal intervals then follows from the fact that for
intervals $M_1$ and $M_2$ of $C$, we have $M_1\subset M_2$ if and
only if $\beta(M_1)\subset \beta(M_2)$.  
\end{proof}

\subsection{Proof of Theorem~\ref{thm:relationships-2}} \label{sec:reln-to-SR-ideals}
We wish to show  that the maps
\begin{align*}
\begin{array}{ccccc}
    \{\text{minimal~primes~of~}I(\Delta(C))\}
        & \xleftarrow{\gamma} &
    \{\text{maximal~codewords~in~}C \}
     &\xrightarrow{\delta} & 
    \left\{ \begin{array}{c}\text{minimal~prime-sets}  \\ \text{of~} \fC(C') \end{array}\right\}
    \\
       \langle x_i \mid i \in [n] \smallsetminus M \rangle 
        & \leftmapsto &
    M
     &\mapsto & 
     \overline{[n] \smallsetminus M}
    \\
\end{array}
\end{align*}
are bijections. The main step is to understand the relationship
between the prime-sets of $\fC(C')$ and the associated primes of $I(\Delta(C))$.

\begin{lemma}
\label{lemma:prime-set}
Let $C$ be a code on $n$ neurons with complement code $C'$.  
A subset $\overline{B}\subset \overline{[n]}$ is a prime-set of
$\fC(C')$ if and only if $\langle x_i \mid i \in B \rangle$
contains $I(\Delta(C))$. Consequently, $\overline{B}$ is a minimal
prime-set of $\fC(C')$ if and only if $\langle x_i \mid i \in B
\rangle$ is a minimal prime of $I(\Delta(C))$.
\end{lemma}

\begin{proof}
By definition, $\overline{B}$ is a prime-set of $\fC(C')$ if and only
if $[n] \cup \overline{B}$ is not a face of $\fC(C')$. 
Equivalently,
every facet of $\fC(C')$ of the form $F=[n]\cup \overline{[n]\smallsetminus
  c}$
  satisfies  $B \cap c \neq \varnothing$.
By Theorem~\ref{thm:relationships-1}, $F=[n]\cup \overline{[n]\smallsetminus
  c}$ is a facet of $\fC(C')$ if and
only if the monomial $\prod_{i\in c}x_i$ belongs to $\CF(J_C)$.
Also, 
$B \cap c \neq \varnothing$ if and only if
$\prod_{j\in c}x_j \in \langle x_i \mid i \in B\rangle$. 
Now the result follows, because the monomials in $\CF(J_C)$ generate $I(\Delta(C))$.
\end{proof}

\begin{proof}[Proof of Theorem~\ref{thm:relationships-2}]
The map $\gamma$ is a bijection, 
by~\eqref{eqn:SRDecomp} and the fact that maximal codewords 
of $C$
are facets of $\Delta(C)$, and $I(\Delta(C))$ is its Stanley--Reisner
ideal.
Given that $\gamma$ is a bijection, Lemma~\ref{lemma:prime-set} shows
that $\delta \circ \gamma^{-1}$ 
is a bijection, and so, $\delta$
is a bijection, completing the proof.
\end{proof}

\section{The \newname complex and the polar complex} \label{sec:polar}
In this section, we explore the relationship between the \newname
complex and the
polar complex
introduced in~\cite{gunturkun2017polarization}. 
For a code $C$, the {\em polar complex}, denoted by $\Delta_\calP(C)$,
is the simplicial complex whose Stanley--Reisner ideal is $\calP(J_C)$, 
the polarization of the neural ideal of $C$.  The ideal $\calP(J_C)$ is 
the {\em polar ideal} of $C$.

We first show in an example that polar and \newname complexes
associated to a code are, in general, not the same.

\begin{example}[Example~\ref{ex:code-6}, continued]\label{polcomp}
For the 
 code $C'=\{ 1, 23, 123 \}$, we polarize the neural ideal 
 $J_{C'} =  \langle (1-x_1)(1-x_3),~ (1-x_1)(1-x_2),~ x_2(1-x_3), x_3(1-x_2)\rangle$
to obtain the polar ideal 
 \begin{align*}
 \calP(J_{C'}) = 
 \langle y_1y_3,~ y_1y_2,~ x_2y_3,~ x_3y_2 \rangle
=
  \langle x_2, x_3, y_1 \rangle \cap \langle y_2, y_3 \rangle \cap \langle x_3, y_1, y_3 \rangle \cap \langle x_2, y_2, y_3 \rangle .
 \end{align*}
It follows that the set of facets of the polar complex $\Delta_{\calP}(C')$
is 
$\{1\bar{2}\bar{3},123\bar{1}, 12\bar{2},13\bar{3}\}$. 
Thus, the polar complex has 2 more facets than the corresponding factor complex (recall Example~\ref{ex:code-4}).
\end{example}

On the other hand, the polar ideal and the \newname ideal (and their
corresponding complexes) share many features. 
A first observation is that $\calP(J_C) \subset \fI(C)$ by construction and Lemma~\ref{lemma:factorIdealPseudomonomials}. 
Furthermore,
Lemma~\ref{lemma:factorIdealPseudomonomials} is valid when we replace
$\fI(C)$ by $\calP(J_C)$
\cite[Theorem 3.2]{gunturkun2017polarization}, and consequently
Lemma~\ref{codewordfacelem} holds for $\Delta_\calP(C)$.
Lemma~\ref{IFCorr} also is valid for $\Delta_\calP(C)$
\cite[Corollary 5.2]{gunturkun2017polarization}.

As Example~\ref{polcomp} illustrates, $\fI(C)$ strictly contains
$\calP(J_C)$ in general.
A larger ideal makes for a smaller simplicial complex. The following
result explains the relationship between $\fC(C)$ and $\Delta_\calP(C)$.

\begin{proposition} \label{prop:factor-polar}
For every code $C$, the \newname complex $\fC(C)$ is the
subcomplex of the polar complex $\Delta_\calP(C)$ whose facets are the
effective facets of $\Delta_\calP(C)$.
\end{proposition}

\begin{proof}
Lemma~\ref{propFulRep} states that all facets of $\fC(C)$ are
effective, and $\calP(J_C)\subset \fI(C)$ implies that $\Delta_{\cap}(C)\subset \Delta_{\calP}(C)$.
So, it suffices to show that every effective facet of
$\Delta_{\calP}(C)$ is a face of $\fC(C)$.
By \cite[Corollaries~5.2~and~5.3]{gunturkun2017polarization}, the
effective facets of $\Delta_\calP(C)$ are of the form
$d\cup\overline{[n]\smallsetminus c}$  where $[c,d]$ is a maximal
interval of $C$. Now apply Lemma~\ref{IFCorr}.
\end{proof}

The key difference between the \newname complex and the polar complex
of a code is that the latter can have defective facets. While these facets
hold useful information about quotient codes, as shown
in~\cite{gunturkun2017polarization}, the structure of the smaller \newname
complex is more convenient for our purposes here.

\section{Computational Considerations}
\label{sec:algorithm}

The main result of this article, Theorem~\ref{thm:main-intro},
gives a new method for checking whether a code is max-intersection-complete (Algorithm~\ref{alg:thealg} below).
In this section we provide an infinite family $\mathfrak{F}$ of codes 
for which this method is more efficient at checking
max-intersection-completeness than the natural brute-force approaches.


In order to analyze the runtime of our proposed algorithm, we write it
explicitly below. Correctness follows directly from
Theorem~\ref{thm:main-intro} and the correspondence between maximal
codewords of $C$ and minimal primes of $I(\Delta(C))$ in Theorem
\ref{thm:relationships-2}.  

%


\begin{algorithm}[H]\label{alg:thealg}
\SetKwInOut{Input}{input}
\SetKwInOut{Output}{output}
{\bf input:}\\
{\begin{itemize}
\item[(1)] $C$, a neural code on $n$ neurons
\item[(2)] $C_{\text{{max}}}$, the list of the maximal codewords of $C$
\item[(3)] $\CF(J_C)$, the canonical form of the neural ideal of $C$
\end{itemize}
}
\Output{\texttt{True} if $C$ is max-intersection-complete and \texttt{False} otherwise}
{\Init Min}$(I(\Delta(C))=\emptyset$;

\For{ {\sc (First Loop)} $c\in C_{\text{\normalfont{{max}}}}$}
{
Add $\langle\{x_i\ |\ i\in [n] \smallsetminus c\} \rangle$ to {Min}$(I(\Delta(C))$;}
\For{ {\sc (Outer Loop)} non-monomial $\phi\in CF(J_C)$} 
{
	\For{{\sc (Middle Loop)} $s$ such that $(1-x_s)|\phi$}
	{
		\For{{\sc (Inner Loop)} $P\in$         
		  \normalfont{{Min}}$(I(\Delta(C))$}
		{
			\If{ $x_s\in P$ and no $x_r\in P$ divides $\phi$} 
			{
                Go back to {\sc Middle Loop} (next iteration of loop, or -- if none -- end loop);
			}
		}
        Go back to {\sc Outer Loop} (next iteration of loop, or -- if none -- end loop);
	}
    \Return \texttt{False};\\
    \EndAl
}
\Return \texttt{True};\\
\EndAl
\caption{Checking Max-Intersection-Completeness}
\end{algorithm}

\begin{remark}
  \label{rmk:Drawback}
  We point out that Algorithm~\ref{alg:thealg} requires $\CF(J_C)$ as part of its input,
  but the brute-force methods below do not. For this reason, a
  complete runtime analysis of Algorithm~\ref{alg:thealg} 
  requires knowing the complexity of computing canonical forms, which
  is not currently well understood. The
  canonical form algorithm given in~\cite{neural-ideal-sage} is easily
  seen to be exponential in the number of neurons. A faster procedure
  for finding $\CF(J_C)$ would be very desirable, and would have
  implications beyond this article.
%
%
%
%
\end{remark}

We now define $\mathfrak{F}$ to be the family of all neural codes $C$
satisfying the following properties:
\begin{itemize}
    \item[(i)] The number of maximal intervals of $C'$ is at most $n$,
      the number of neurons of $C$.
    \item[(ii)] There exists a maximal interval $[c,d]$ of $C'$ with $d\neq [n]$ and $|d\smallsetminus c | = n/2$.
    \item[(iii)] There exists a maximal interval $[a,[n]]$ of $C'$, where $a$ contains $n/2$ neurons.
    \item[(iv)] For every maximal interval of $C'$ that has the form
      $[b,[n]]$, if $a\neq b$ then $a\cap b = \emptyset $. 
    \item[(v)] $C'$ contains at most $\log_2(n)$ maximal intervals of the form $[b,[n]]$.
\end{itemize}

Note that $\mathfrak{F}$ is infinite, since the number of neurons has not
been fixed.
We emphasize that a code $C \in \mathfrak{F}$ is given as the maximal
intervals of $C'$. This information is equivalent to knowing
$\CF(J_C)$. Thus, for codes in $\mathfrak{F}$, the issue raised in
Remark~\ref{rmk:Drawback} is avoided.
Finally, it can be checked that $\mathfrak{F}$ contains infinitely
many max-intersection-complete codes, and infinitely many codes which
are not max-intersection-complete.

We compare Algorithm~\ref{alg:thealg} to two brute-force methods for checking max-intersection-completeness:

\textbf{Brute Force 1:} Take all possible intersections of maximal
codewords of $C$, and check whether all are contained in $C$. 

\textbf{Brute Force 2:} For every $\sigma \in C'$, compute
$c_{\sigma}$, the intersection of all maximal codewords of $C$ that
contain $\sigma$. Then check whether $c_{\sigma} = \sigma$.  

\begin{proposition} \label{prop:alg}
For every code $C$ in $\mathfrak{F}$, Brute Force 1 and Brute Force 2
are exponential in the number of neurons, while Algorithm
\ref{alg:thealg} is sub-exponential in the number of neurons.  
\end{proposition}

\begin{proof}
  We begin by showing that the number of maximal codewords of any $C
  \in \mathfrak{F}$
is at least $n/2$ and at most $n^{\log_2(n)}$.
Recall that these maximal codewords are in bijection with the minimal
primes of $I(\Delta(C))$ (Theorem \ref{thm:relationships-2}),  
and also that 
\begin{equation}\label{eqn:monomial_gens}
I(\Delta(C)) = \langle \{x_{\sigma} \ |\ [\sigma,[n]] \text{ a maximal interval of } C' \} \rangle.
\end{equation}
(Recall that for $\sigma\subseteq [n]$, we use the notation $x_{\sigma}$ to denote the monomial $\prod_{i\in \sigma} x_i$.)

The monomial generators of $I(\Delta (C))$ in~\eqref{eqn:monomial_gens} satisfy the following:
\begin{itemize}
    \item[($\ast$)] there is a generator $x_a$ of degree $n/2$ (from condition (iii)),
    \item[($\ast \ast$)] if $x_b \neq x_a$ is a generator of $I(\Delta(C))$, then $\gcd(x_a,x_b) = 1$ (from condition (iv)), and
    \item[($\ast \ast \ast$)] there are at most $\log_2(n)$ generators (from condition (v)).
\end{itemize}

We calculate the upper bound by observing that every minimal prime $P$
of $I(\Delta(C))$ has a generating set $G_P \subset
\{x_1,x_2,\dots,x_n\},$ with every monomial in
\eqref{eqn:monomial_gens} divisible by at least one $x_i\in
G_P$. It follows that the number of ways to choose
some divisor $x_i$ from each generator of $I(\Delta(C))$ is an upper
bound on the number of minimal primes. This upper bound is the product
of the degrees of the monomial generators of $I(\Delta(C))$, which in
turn is bounded above by $n^{N_{\text{mon}}}$, where $N_{\text{mon}}$
is the number of monomials in $CF(J_C)$. By ($\ast \ast \ast$) there
are at most $\log_2(n)$ such monomials, so  
the number of minimal primes -- and thus the number of maximal codewords of $C$ -- is at most $n^{\log_2(n)}$.

For the lower bound, we first note that by ($\ast$) there is a
monomial generator $x_a$ of $I(\Delta(C))$ that has degree $n/2$. If
$I(\Delta(C)) = \langle x^a \rangle$, then $I(\Delta(C))$ has $n/2$
minimal primes.
If $I(\Delta(C))$ strictly contains $\langle x^a \rangle$, then let $\widetilde{P}$ be a minimal
prime of the following nonzero ideal:
\[
\widetilde{I} ~:=~ \langle x_b  \mid  [b,[n]] \text{ is a maximal
      interval of } C' \text{ and } b\neq a  \rangle ~ \subset ~
    I(\Delta(C)).
\]
For every $x_i$ that divides $x_a$, we  
claim that
$P_i = \langle x_i\rangle + \widetilde{P}$ 
is a minimal prime of $I(\Delta(C))$. 
By construction, $P_i$ contains $I(\Delta(C))$.
If $Q\subsetneq P_i$ is another prime monomial ideal, either $x_i\not\in Q$ or 
there exists $x_j \in \widetilde P \setminus Q$. 
In the first case, condition ($\ast \ast$) implies that $x_a\not\in
Q$.  In the second case, by ($\ast \ast$) and the fact that
$\widetilde P$ is a minimal prime of $\widetilde I$, it follows that $\widetilde I \not\subset Q$.
In both cases $Q$ does not contain $I(\Delta(C))$, and consequently $P_i$ is a minimal prime of $I(\Delta(C))$. As a distinct minimal prime $P_i = \langle x_i\rangle + \widetilde{P}$ arises from each of the $n/2$ divisors $x_i$ of $x_a$, 
the number of minimal primes -- and also the number of maximal codewords of $C$ -- is at least $n/2$.

Having found the upper and lower bounds on the number of maximal codewords of a code $C\in \mathfrak{F}$, 
we now use these bounds to 
analyze the brute-force methods and Algorithm \ref{alg:thealg}.

As there are at least $n/2$ maximal codewords, Brute Force 1 checks at least $2^{n/2}$ intersections of maximal codewords, and so is exponential in the number of neurons. 

Next, Brute Force 2 checks whether each codeword of $C'$ 
is contained in each maximal codeword of $C$.  So, the runtime will be at least the number of codewords of $C'$ times the number of maximal codewords of $C$. 
There are at least $n/2$ maximal codewords and, by condition (ii), at least $2^{n/2}$ elements of $C'$. Thus, the runtime is at least $(n/2)*2^{n/2}$, and so is exponential in $n$.

For Algorithm \ref{alg:thealg}, 
First Loop iterates over the maximal codewords of $C$ (of which there are at most $n^{\log_2(n)}$), and the runtime of each iteration is at most $n$.  So, the runtime of First Loop is $O(n^{1+\log_2(n)})$.
The runtime for the  subsequent part of the algorithm is the product
of the number of iterations of the Outer Loop, the number of
iterations of the Middle Loop, and the runtime of the Inner
Loop. Since the Outer Loop iterates over a subset of $\CF(J_C)$, by
Theorem \ref{thm:relationships-1} and condition (i) there are at most
$n$ such iterations. Since the Middle Loop iterates over the neurons,
there are at most $n$ iterations of this loop. Finally, the Inner Loop
iterates over the number of minimal primes of $I(\Delta(C))$, of which
there are at most $n^{\log_2(n)}$. Checking to see if $x_s$ is in some
minimal prime $P$ takes at most $n$ steps (check each generator of
$P$) and checking to see if any $x_r\in P$ divides $\phi$ takes
at most $n^2$ steps (compare each generator of $P$ with each divisor
of $\phi$). Thus the runtime of Inner Loop is at most
$n^{3+\log_2(n)}$. We conclude that the combined runtime of the Outer, Middle, and Inner Loops is $O(n^{5+\log_2(n)})$, 
which, it is straightforward to check, is sub-exponential in $n$.
\end{proof}

%



\bibliographystyle{siam}
\bibliography{Bibliography.bib}



\end{document}